\documentclass{amsart}
\usepackage{color}

\newcommand{\NN}{\ensuremath{\mathbb{N}}}

\newcommand{\RR}{\ensuremath{\mathbb{R}}}

\newcommand{\TT}{\ensuremath{\mathbb{T}}}

\newcommand{\ZZ}{\ensuremath{\mathbb{Z}}}

\newcommand{\eps}{\ensuremath{\epsilon}}

\newcommand{\aA}{\ensuremath{\mathcal{A}}}

\newcommand{\xX}{\ensuremath{\mathcal{X}}}
\newcommand{\yY}{\ensuremath{\mathcal{Y}}}

   \newtheorem{lemma}{Lemma}[section]
   \newtheorem{theorem}[lemma]{Theorem}
   \newtheorem{remark}[lemma]{Remark}

   \newtheorem{definition}[lemma]{Definition}

\numberwithin{equation}{section}
\newcommand{\e}{\epsilon}

\newcommand{\R}{{\mathbb R}}

\begin{document}


\title[Longtime behavior for 3D Navier-Stokes equations with constant delays]
{Longtime behavior for 3D Navier-Stokes equations with constant delays}

\author{Hakima Bessaih}
 \address[Hakima Bessaih]{Department of Mathematics and Statistics\\
 University of Wyoming\\
Laramie 82071 USA}
\email[Hakima Bessaih]{bessaih@uwyo.edu}

\author{Mar\'{\i}a J. Garrido-Atienza}\address[Mar\'{\i}a J. Garrido-Atienza]{Dpto. Ecuaciones Diferenciales y An\'alisis Num\'erico\\
Facultad de Matem\'aticas, Universidad de Sevilla, Avda. Reina Mercedes, s/n, 41012-Sevilla, Spain} \email[Mar\'{\i}a J. Garrido-Atienza]{mgarrido@us.es}

\begin{abstract} 
This paper investigates the longtime behavior of delayed 3D Navier-Stokes equations in terms of attractors. The study will strongly rely on the investigation of the linearized Navier-Stokes system, and the relationship between the discrete dynamical flow for the linearized system and the continuous flow associated to the original system. Assuming the viscosity to be sufficiently large, there exists a unique local attractor for the delayed 3D Navier-Stokes equations. Moreover, the local attractor reduces to a singleton set.
\end{abstract}

\maketitle

\section*{}

\section{Introduction}\label{s1} 
The incompressible 3D Navier-Stokes equations are described by time evolution of the velocity $u$ in a bounded or unbounded domain of $\mathbb{R}^3$  and are given by:
\begin{align*}
 u^\prime(t,x) & +(u(t,x)\cdot\nabla) u(t,x)-\nu\Delta u(t,x)+\nabla p(t,x)=0, \\
   & {\rm div}\, u(t,x)=0,\quad u(0,x)=u_0(x),
\end{align*}
where $\nu>0$ is the viscosity of the fluid, $p$ denotes the pressure and $u_0(x)$ denotes the initial datum.  The uniqueness of global weak solutions is a standing open problem. In order to overcome this challenging difficulty, in a previous paper, see \cite{BGSch17},  we introduced a constant delay $\mu>0$ into the nonlinear term   $(u\, \cdot\, \nabla) u$. More precisely, we considered the following modified version of the 3D Navier-Stokes equations: 
 
\begin{align}\label{delay-i}
\begin{split}
  u^\prime(t,x) & +(u(t-\mu, x)\cdot\nabla)u(t,x)-\nu\Delta u(t,x)+\nabla p(t,x)=f(x), \\
   & {\rm div}\, u(t,x)=0,\quad u(0,x)=u_0(x),\quad u(\tau,x)=\phi(\tau,x),\quad \tau\in [-\mu,0).
\end{split}
\end{align}

This delay introduces a regularizing  effect in the equations and allows to prove the uniqueness of global weak solutions when the 
initial function $(\phi,u_0)\in L_2(-\mu,0,V^{1+\alpha})\times V^\alpha$ with $\alpha>1/2$ (for the definition of the spaces $V^\alpha$ see Section \ref{s2}). In particular, when $\alpha\geq1$, then our theory can be extended to include strong solutions. The main ingredient to establish it is to use the regularizing effect of the delay on the convective term by investigating the linearized version of \eqref{delay-i}. This equation comes naturally when investigating  the system on the interval $[0,\mu] $. We prove existence and uniqueness of weak solutions, then we establish that these solutions are  more regular and are in the spaces $V^\alpha$. Then, we use a concatenation argument by glueing the solutions obtained on each interval $[0,\mu], [\mu, 2\mu],\dots $ and so on. Each solution is obtained from the previous step and uses the linearized construction. 

As a byproduct, the linearized equation induces a continuous mapping $U$ on the space $L_2(0, \mu,V^{1+\alpha})\times V^\alpha$. The $n$th composition of the map $U$ generates a discrete semigroup $U(n)$ on the same space $L_2(0, \mu,V^{1+\alpha})\times V^\alpha$. Moreover, thanks to the concatenation argument the solution of  system \eqref{delay-i} generates a continuous semigroup 
$S(t)$ on the space  $L_2(-\mu,0,V^{1+\alpha})\times V^\alpha$ for $t\geq 0$
given by $S(t)(\phi,u_0)=(u_t^\mu, u^\mu(t))$, 
where $u_t^\mu$ is the segment function defined by $u_t^\mu(s)=u^\mu(t+s)$, $s\in (-\mu,0)$ ($u^\mu$ denotes the solution of (\ref{delay-i})), defined in more details in Section  2.

Our goal in this paper is to study the longtime behavior of (\ref{delay-i}) in terms of attractors. Let us point out that the existence of a local attractor for (\ref{delay-i}) is essentially based on getting an invariant ball for the map $U$.  The main ingredient that allows to get the invariance of a bounded ball for the discrete semigroup $U(n)$ and later the semigroup $S(t)$ is the fact that the unique weak global solution is more regular, see Lemma 3.1, 3.2 and 3.3. Combined with compactness embeddings, this allows to prove that the local invariant ball for the discrete semigroup $U(n)$ is compact in the topology of $L_2(0, \mu,V^{1+\alpha})\times V^\alpha$. 
An interesting feature of this model is that we are able to establish that the local attractor $\mathcal{A}$ associated to the discrete semigroup $U(n)$ is a single point attractor. These properties are transferred to the original delayed 3D Navier-Stokes equations, due to the key relationship between $U$ and $S$. In fact, under the same conditions as for the discrete flow $U$, the continuous flow $S$ is proved to have a local attractor $\mathcal{A}^\mu$, that reduces to a singleton set and is linked to the local attractor $\mathcal{A}$ as $S(t)$ and $U$ are related on the grid points $t=n\mu$, $n\in \NN$.\\

The paper is organized as follows. In Section 2 we introduce the abstract setting in which we develop our theory and recall how the construction of the unique weak solution of (\ref{delay-i}) was carried out in the paper \cite{BGSch17}, by using a suitable linearization of (\ref{delay-i}) on $[0,\mu]$. Section 3 addresses the regularization properties of the solution of (\ref{delay-i}) assuming that the external force is in $V^\alpha$. In Section 4, we first consider a linearized system defined now on any compact interval $[0,T]$ for a given $T>0$ and construct its corresponding unique weak solution. Then we establish a fundamental relationship between the discrete flow $U$ generated by the solution of the linearized system and the continuous flow $S$ generated by the solution of (\ref{delay-i}). Section 5 is devoted to the study of the local attractor for the linearized system and finally, in Section 6, we establish the existence of a unique local attractor for $S$ and we study its inner structure. 

\section{Preliminaries: existence and uniqueness of a weak solution}\label{s2}

We introduce in this section the functional setting in which our investigations will be carried out and the existence and uniqueness of solutions of the delayed Navier-Stokes equations as well. 

Consider the torus $\TT^3_L$ in $\R^3$ of length $L$ given by the set 
$$\TT^3_L:=\{ x=(x_1,x_2,x_3)\in \RR^3 \,:\, -L/2\leq x_i \leq L/2; \, x_i=-L/2 \mbox{ is identified with  } x_i=L/2, \, i=1,2,3\}.$$ Let $\psi(x)$ be a $L$-periodic function that can be expanded into Fourier series
$$\psi(x)=\sum_{\zeta\in \ZZ^3_L} e^{i(x,\zeta)} \hat \psi(\zeta),$$
where
$$\ZZ^3_L=\{ \zeta=(\zeta_1,\zeta_2,\zeta_3)\,:\, \zeta_i=2\pi k_i/L,\, k_i \mbox { is an integer}, \, i=1,2,3\},$$
and
$$\hat \psi(\zeta)=L^{-3} \int_{\TT^3_L} e^{-(y,\zeta)}\psi(y)dy$$
denote the Fourier coefficients of $\psi$.

For $s\in \RR$, we denote by $H^s(\TT^3_L)$ the Sobolev space of $L$--periodic functions such that $\hat \psi(\zeta)=\overline{\hat \psi(-\zeta)}$ equipped with the norm
$$\|\psi\|_s= \bigg(\sum_{\zeta\in \ZZ^3_L} (1+|\zeta|^2)^s |\hat \psi(\zeta)|^2\bigg)^\frac12.$$
When $\hat \psi(0)=0$ the corresponding subspace is denoted by $\dot H^s(\TT^3_L)$ with equivalent norm
$$\bigg(\sum_{\zeta\in \ZZ^3_L\setminus \{0\}} |\zeta|^{2s} |\hat \psi(\zeta)|^2\bigg)^{1/2}.$$
These spaces are Hilbert--spaces
with the inner product
\begin{equation*}
  (\psi_1,\psi_2)_s =\sum_{\zeta\in \ZZ^3_L\setminus \{0\}}|\zeta|^{2s} \psi_1(\zeta)\overline{\hat\psi_2(\zeta)}.
\end{equation*}
\medskip
We denote $\dot{\mathbb H}^s(\mathbb T_L^3)=\dot H^s(\TT^3_L)^3$ and, for $s=-1,0,1$, we introduce the spaces
\begin{align*}
V^s=&\{u\in \dot{\mathbb H}^s(\mathbb T_L^3), \, \rm{div}\ u=0 \}.
\end{align*}
Then $V^{-1}$ is the dual space of $V^1$ and $V^1\subset V^0 \subset V^{-1}$ where the injections are continuous and each space is dense in the following one. We shall denote by $(\cdot,\cdot)$ the scalar product in $V^0$.

We introduce the Stokes operator $A$ as in \cite{temamP}, Section 2.2, page 9, with domain given by
$$D(A)=\{u\in V^0, \, \Delta u\in V^0 \}.$$
For the periodic boundary conditions we know that
\begin{equation*}
  Au=-\Delta u.
\end{equation*}
The operator $A$ can be seen as an unbounded positive linear selfadjoint operator on $V^0$, and we can define the powers $A^s$, $s\in \RR$ with domain $D(A^s)$. We set $V^{s}=D(A^{s/2})$, that is a closed subspace of $\dot {\mathbb H}^s(\TT_L^3)$, then for any $s\in \RR$
$$V^s=\{u\in \dot{\mathbb H}^s(\mathbb T_L^3), \, \rm{div}\ u=0 \}$$
and the norms $\|A^{s/2}u\|_0$ and $\|u\|_s$ are equivalent on $V^s$.  The operator $A$ defines an isomorphism from $V^s$ to $V^{s-2}$, and has a positive countable spectrum of finite multiplicity $0<\lambda_1\le \lambda_2\le \cdots$, $\lambda_j \to \infty$, where the associated eigenvectors $e_1,e_2,\cdots$ form a complete orthogonal system in $V^s$. 

When  $s_1<s_2$,  the embedding $V^{s_2}\subset V^{s_1}$ is compact and dense. The space $V^{-s}$ is the dual space of $V^s$ for $s\in\RR$, see Temam \cite{temamP}, from page 9. We shall denote by $\langle \cdot, \cdot \rangle$ the duality product between $V^s$ and $V^{-s}$ no matter the value of $s\in \RR$.\\

Let us introduce the trilinear form $b$ given by
\begin{equation*}
  b(u,v,w)=\sum_{i,j=1}^3\int_{\TT_L^3} u_j \frac{\partial v_i}{\partial x_j}w_idx.
\end{equation*}

The following result is essential in our estimates. For the proof, we refer to \cite{BGSch17}.

\begin{lemma}\label{l2}
The trilinear form $b$ can be continuously  extended to $V^{s_1}\times V^{s_2+1}\times V^{s_3}$  for $s_i\in\RR$ if
either $s_i+s_j\ge 0$ for $i\not=j$, $s_1+s_2+s_3>3/2$ or $s_i+s_j> 0$ for $i\not=j$, $s_1+s_2+s_3\ge 3/2$. Therefore, under either of the previous settings, there exists a constant $c$ depending only on $s_i$ such that
\begin{equation*}
  |b(u,v,w)|\le c\|u\|_{s_1}\|v\|_{s_2+1}\|w\|_{s_3}
\end{equation*}
for $u\in V^{s_1},\,v\in V^{s_2+1},\,w\in V^{s_3}$.
\end{lemma}

Notice that similar results were proved by Fursikov \cite{Fursikov} when considering a bounded domain $\Omega \subset \RR^3$, $\partial \Omega \in C^\infty$ with homogeneous Dirichlet conditions, but with more restrictive assumptions. In the periodic boundary setting, for a similar result as Lemma \ref{l2} above see also Temam \cite{temamP}, Lemma 2.1, page 12, which holds true under the additional assumptions $s_i\geq 0$.

When $u,\,v,\,w\in V^1$ it is known that $ b(u,v,w)=-b(u,w,v)$, which implies $b(u,v,v)=0.$ Furthermore, from the trilinear form $b$ we can derive a bilinear operator $B:V^{s_1}\times V^{s_2+1} \rightarrow V^{-s_3}$ given by
\begin{align*}
& \langle B(u,v),w\rangle = b(u,v,w),
\end{align*}
such that
\begin{align}\label{eq7}
\begin{split}
& \|B(u,v)\|_{-s_3}\le c \|u\|_{s_1}\|v\|_{s_2+1}
\end{split}
\end{align}
with $s_1,\,s_2,\,s_3$ satisfying the conditions of  Lemma \ref{l2}. \medskip

Finally we mention that for $\mu>0$ and $s\in \RR$ the spaces $L_\infty(0,\mu,V^s),\,L_2(0,\mu,V^s),\,C([0,\mu],V^s)$ and $C^\beta([0,\mu],V^s),\,\beta\in (0,1)$, have
the usual meanings. 

We are interested in studying the dynamics of the following version of the 3D Navier-Stokes equations with constant delay $\mu$:
\begin{align}\label{delay}
\begin{split}
  u^\prime(t,x) & +(u(t-\mu),\nabla)u(t)-\nu\Delta u(t)+\nabla p(t,x)=f(x), \\
   & {\rm div}\, u(t,x)=0,\quad u(0)=u_0(x),\quad u(t)=\phi(t),\quad t\in [-\mu,0).
\end{split}
\end{align}
Denote the solution of this equation depending on the time shift by $u^\mu$. On account of the Helmholtz-projection,  we can formulate the equation as
 \begin{equation}\label{D2}
  \left\{
  \begin{aligned}
    du^\mu(t)&+(\nu Au^\mu(t)+B(u^\mu(t-\mu), u^\mu(t)))dt=fdt ,\,&t \geq 0,&\\
    u^\mu(0)&=u_{0},\\
    u^\mu(t)&=\phi (t),\,&t\in[-\mu,0).
  \end{aligned}
  \right.
\end{equation}
\begin{definition}\label{def-deb}
Let $\mu>0$ and $\alpha>1/2$.
We are given $u_0\in {V^\alpha}, \phi \in L^2(-\mu, 0, V^{1+\alpha})$,  $f\in V^{{\alpha-1}}$ and $T>0$.
We say that $u^\mu$ is a weak solution to system \eqref{D2}
on the time interval $[-\mu,T]$ if
\begin{align*}
u^\mu\in L_2(-\mu,T, V^{1+\alpha}),
\end{align*}
with $u^\mu(0)=u_0$, $u^\mu(t)=\phi (t)$ for $t\in[-\mu,0)$, and, given any $v\in V^{\alpha+1}$ and any test function $\varphi \in  C_0^\infty([0,T])$,
\begin{align}\label{var-v}
  -\int_0^T \langle u^\mu(r),v \rangle \varphi^\prime(r)dr  &+\nu\int_{0}^{T}\langle A^{1/2}u^\mu(r), A^{1/2} v\rangle \varphi(r) dr +\int_{0}^{T}\langle B(u^\mu(r-\mu),u^\mu(r)),v \rangle \varphi(r) dr
\nonumber\\
&= \int_{0}^{T}\langle f,v \rangle \varphi(r)dr.
\end{align}
\end{definition}

In order to prove the existence and uniqueness of solutions to (\ref{var-v}), for $t \in [0,\mu]$ and $\psi \in  L_2(0,\mu,V^{1+\alpha})$, we introduce the following 3D linearized Navier--Stokes equations with periodic boundary conditions over the torus  $\TT_L^3$ in $\RR^3$
 \begin{equation}\label{eq1}
  \left\{
  \begin{aligned}
    du(t)&+(\nu Au (t)+B(\psi(t), u(t)))dt=fdt ,\quad&t \in [0,\mu],&\\
    u(0)&=u_{0}.
  \end{aligned}
  \right.
\end{equation}
These equations are a simpler version of the 3D Navier--Stokes equations, since the term $(u,\nabla)u$ has been replaced by $(\psi,\nabla)u$. The existence and uniqueness of solutions to (\ref{D2}) and (\ref{eq1}) can be summarized as follows. 

For the sake of readability, the solutions of (\ref{D2}) will be denoted by $u^\mu$ while the corresponding solutions to  (\ref{eq1}) are denoted by $u$.

\begin{theorem}\label{t1}
Assume that $u_0\in V^\alpha$ and $f\in V^{\alpha-1}$. Then
\begin{enumerate}
\item If $\psi \in  L_2(0,\mu,V^{1+\alpha})$, \eqref{eq1} has a weak solution $u\in L_\infty(0,\mu,V^\alpha)\cap L_2(0,\mu,V^{1+\alpha})\cap C([0,\mu],V^\alpha)$.
\item If $\phi\in L_2(-\mu,0,V^{1+\alpha})$, there exists a unique weak solution $u^\mu$ to \eqref{var-v} in the sense of Definition \ref{def-deb}. Furthermore, $u^\mu|_{[0,T]}\in C([0,T], V^\alpha)$ and, if $0\leq \gamma \leq 1/2$ and $s\geq 1$, we also obtain $u^\mu|_{[0,T]}\in L_\infty(0,T,V^\alpha)\cap C^\gamma([0,T],V^{-s})$, and $\frac{du^\mu}{dt}\in L_2(0,T,V^{\alpha-1})$.
\end{enumerate}
\end{theorem}
\begin{proof}
Although the proof of this theorem is in the paper \cite{BGSch17}, for the sake of completeness we would like to give here some explanations of how to prove this result. The existence of a weak solution for the linearized problem (\ref{eq1}) is obtained thanks to the use of Galerkin approximations, while the uniqueness relies on an energy equality, based on the fact that $u^\prime \in L_2(0,\mu,V^{-1})$.

To prove existence and uniqueness of a weak solution of (\ref{eq1}) on $[-\mu,T]$, the strategy followed in \cite{BGSch17} consists of solving the problem (\ref{D2}) step by step, in intervals of length $\mu$, where in each step it is used the fact that for (\ref{eq1}) there exists a unique weak solution. As a result, a sequence $\{(u_k^\mu)\}_{k\in \NN}\subset L_2(-\mu,\mu,V^{1+\alpha})$ is built, with $u^\mu_k(\mu)\in V^\alpha$ and $u^\mu_k(\cdot-\mu)=u^\mu_{k-1}(\cdot)\in L_2(0,\mu,V^{1+\alpha})$,  for any $k\in \NN$. Concatenating the elements of the sequence the global solution of (\ref{D2}) is constructed, having the following expression 
\begin{align}
\label{sol}
u^\mu(t)=\left\{\begin{array}{lll}
    \phi(t)&\mbox{ if}& t\in[-\mu,0),\\
           u_0&\mbox{ if}& t=0,\\
        u^\mu_1(t)&\mbox{ if}& t\in [0,\mu],\\
     u^\mu_2(t-\mu)&\mbox{ if}& t\in [\mu,2\mu],\\
     \vdots&&\vdots \\
      u^\mu_k(t-(k-1)\mu)&\mbox{ if}& t\in [(k-1)\mu,T],
    \end{array}
    \right.
\end{align}
assuming that $t\in [0,T]$, with $T\in  ((k-1)\mu, k\mu].$

\end{proof}

\section{Regularization of weak solutions}\label{s3}

In this section, we are going to show that assuming $f\in V^\alpha$ then the solution $u^\mu$ to (\ref{D2}) is more regular. Thanks to this regularity, we will obtain a suitable compact property that will be further necessary to establish the existence of an attractor for the delayed Navier--Stokes equations.\\

From now on, we denote $\yY^\mu_\alpha=L_2(-\mu,0,V^{1+\alpha})\times V^\alpha$.

\begin{lemma}\label{a1}
Assume that $(\phi, u_0) \in \mathcal{Y}^\mu_\alpha$ and $f\in V^\alpha$. Then $u^\mu(t)\in V^{1+\alpha}$, for $t>0$.
\end{lemma}
\begin{proof}
Let us first assume that $t\in (0,\mu]$. Considering the scalar product with $A^{1+\alpha} u^\mu(t)$ in $V^0$, it can be derived that 
\begin{align*}
\frac{d}{dt}(t\|u^\mu(t)\|^2_{1+\alpha})&=\|u^\mu(t)\|^2_{1+\alpha}+2t \langle \frac{du^\mu}{dt}(t),A^{1+\alpha}u^\mu(t)\rangle\\
&\leq \|u^\mu(t)\|^2_{1+\alpha}-2t \langle A u^\mu(t), A^{1+\alpha} u^\mu(t)\rangle- 2t \langle B(\phi(t-\mu),u^\mu(t)), A^{1+\alpha} u^\mu(t)\rangle +2t \langle f, A^{1+\alpha} u^\mu(t)\rangle\\
&\leq \|u^\mu(t)\|^2_{1+\alpha}-2t \|u^\mu(t)\|_{2+\alpha}^2+2t \|\phi(t-\mu)\|_{s_1}\|u^\mu(t)\|_{s_2+1}\|u^\mu(t)\|_{2+2\alpha+s_3}+2t \|f\|_\alpha \|u^\mu(t)\|_{2+\alpha}\\
&\leq \|u^\mu(t)\|^2_{1+\alpha}+t \|\phi(t-\mu)\|_{1+\alpha}^2\|u^\mu(t)\|_{1+\alpha}^2+t \|f\|_\alpha^2,
\end{align*}
where we have applied Lemma \ref{l2} with $s_1=1+\alpha$, $s_2=\alpha$ and $s_3=-\alpha$ (we remind here that in all the paper $\alpha>1/2$). As a consequence,
\begin{align*}
t\|u^\mu(t)\|^2_{1+\alpha}&\leq  \|u^\mu\|^2_{L_2(0,\mu,V^{1+\alpha})}+\int_{0}^{t}s \|f\|_\alpha^2ds+\int_0^t \|\phi(s-\mu)\|_{1+\alpha}^2 (s \|u^\mu(s)\|_{1+\alpha}^2)ds,
\end{align*}
and in virtue of Gronwall's lemma,
\begin{align*}
t\|u^\mu(t)\|^2_{1+\alpha}&\leq  (\|u^\mu\|^2_{L_2(0,\mu,V^{1+\alpha})}+t \|f\|_\alpha^2) e^{\|\phi\|^2_{L_2(-\mu,0,V^{1+\alpha})}}.
\end{align*}
If now $t\in [\mu,2\mu]$, we can repeat similar steps than before to arrive at
\begin{align*}
\frac{d}{dt}(t\|u^\mu(t)\|^2_{1+\alpha})&\leq \|u^\mu(t)\|^2_{1+\alpha}-2t \langle A u^\mu(t), A^{1+\alpha} u^\mu(t)\rangle- 2t \langle B(u^\mu(t-\mu),u^\mu(t)), A^{1+\alpha} u^\mu(t)\rangle +2t \langle f, A^{1+\alpha} u^\mu(t)\rangle\\
&\leq \|u^\mu(t)\|^2_{1+\alpha}+t \|u^\mu(t-\mu)\|_{1+\alpha}^2\|u^\mu(t)\|_{1+\alpha}^2+t \|f\|_\alpha^2,
\end{align*}
and integrating
\begin{align*}
t\|u^\mu(t)\|^2_{1+\alpha}&\leq  \mu \|u^\mu(\mu)\|^2_{1+\alpha}+\|u^\mu\|^2_{L_2(\mu,2\mu,V^{1+\alpha})}+(t-\mu) \|f\|_\alpha^2+\int_\mu^t \|u^\mu(s-\mu)\|_{1+\alpha}^2 (s \|u^\mu(s)\|_{1+\alpha}^2)ds,
\end{align*}
hence
\begin{align*}
t\|u^\mu(t)\|^2_{1+\alpha}&\leq  (\mu (\|u^\mu(\mu)\|^2_{1+\alpha} +\|f\|_\alpha^2)+\|u^\mu\|^2_{L_2(\mu,2\mu,V^{1+\alpha})})e^{\int_\mu^t \|u^\mu(s-\mu)\|_{1+\alpha}^2 ds}\\
&= (\mu (\|u^\mu(\mu)\|^2_{1+\alpha} +\|f\|_\alpha^2)+\|u^\mu\|^2_{L_2(\mu,2\mu,V^{1+\alpha})})e^{\|u^\mu\|^2_{L_2(0,\mu,V^{1+\alpha})}}.
\end{align*}
It is clear that due to the regularity of the weak solution $u^\mu$ we can repeat this procedure in any interval. This completes the proof.
\end{proof}

We can also establish the following regularity result:
\begin{lemma}\label{a2}
Assume that $(\phi, u_0) \in \mathcal{Y}^\mu_\alpha$ and $f\in V^\alpha$. Then for every $\eps>0$, the solution of (\ref{D2}) satisfies  $u^\mu \in L_\infty(\eps,T,V^{1+\alpha})\cap L_2(\eps,T,V^{2+\alpha})$.
\end{lemma}
\begin{proof}
The proof is based on the regularity properties of the weak solution together with the fact that, as a consequence of Lemma \ref{a1}, we know that for any $\eps>0$
$$\sup_{t\in[\eps,T]}\|u^\mu(t)\|_{1+\alpha }^2 <\infty.$$
Indeed, as in the previous proof, assume first that $t\in [0,\mu]$. Then
\begin{align*}
  \frac{d}{dt}\|u^\mu(t)\|_{1+\alpha}^2+2\nu\|u^\mu(t)\|_{2+\alpha}^2&\le 2c\|u^\mu(t-\mu)\|_{1+\alpha }\|u^\mu(t)\|_{1+\alpha }\|A^{1+\alpha}u^\mu(t)\|_{-\alpha} +\frac{2}{\nu}\|f\|_{\alpha}^2+\frac{\nu}{2}\|u^\mu(t)\|_{2+\alpha}^2\\
  &\le \frac{c^2}{\nu}\|u^\mu(t-\mu)\|_{1+\alpha }^2\|u^\mu(t)\|_{1+\alpha }^2+\nu\|u^\mu(t)\|_{2+\alpha}^2+ \frac{2}{\nu}\|f\|_{\alpha }^2
  +\frac{\nu}{2}\|u^\mu(t)\|_{2+\alpha}^2.
\end{align*}
Above, to estimate the trilinear form, we have taken in Lemma \ref{l2} the parameters $s_1=1+\alpha$, $s_2=\alpha$ and $s_3=-\alpha$.

Hence, by integration,
\begin{align*}
\|u^\mu(t)\|_{1+\alpha}^2+\frac{\nu}{2}\int_\eps ^t \|u^\mu(s)\|_{2+\alpha}^2ds  &\le \|u^\mu(\eps)\|_{1+\alpha}^2 +\frac{c^2}{\nu} \sup_{t\in [\eps,\mu]}\|u^\mu(t)\|_{1+\alpha }^2 \|\phi\|_{L_2(-\mu,0,V^{1+\alpha}) }^2+ \frac{2}{\nu}(t-\eps)\|f\|_{\alpha }^2
\end{align*}
which implies $u^\mu \in L_2(\eps,\mu,V^{2+\alpha})$. Reasoning in a similar way, when $t\in[\mu,2\mu]$, we have
\begin{align*}
\|u^\mu(t)\|_{1+\alpha}^2+\frac{\nu}{2}\int_\mu ^t \|u^\mu(s)\|_{2+\alpha}^2ds  &\le \|u^\mu(\mu)\|_{1+\alpha}^2 +\frac{c^2}{\nu} \sup_{t\in [\mu,2\mu]}\|u^\mu(t)\|_{1+\alpha }^2 \|u^\mu\|_{L_2(0,\mu,V^{1+\alpha}) }^2+ \frac{2}{\nu}(t-\mu)\|f\|_{\alpha }^2
\end{align*}
hence $u^\mu \in L_2(\mu,2\mu,V^{2+\alpha})$. Repeating the same argument we conclude the proof.
\end{proof}

We can also establish the H\"older regularity of the solution.
\begin{lemma}\label{a3}
Assume that $(\phi, u_0) \in \mathcal{Y}^\mu_\alpha$ and $f\in V^\alpha$. Then for every $\eps>0$, the solution of (\ref{D2}) satisfies $u^\mu \in C^\beta([\eps,T],V^{\alpha})$ for $\beta \in (0,1/2]$. 
\end{lemma}
\begin{proof}
Consider $\eps \leq s <t\leq T$. Then
$$\|u^\mu (t)-u^\mu(s)\|_\alpha \leq \nu (t-s)^{1/2} \bigg(\int_\eps^T \|Au(r)\|_\alpha^2 dr \bigg)^{1/2}+ (t-s)^{1/2}\bigg(\int_\eps^T \|B(u^\mu(r-\mu),u(r))\|_\alpha^2 dr \bigg)^{1/2}.$$
On the one hand,
$$\int_\eps^T \|Au(r)\|_\alpha^2 dr =\int_\eps^T \|A^{\frac{\alpha}{2}}Au(r)\|^2_0 dr=\int_\eps^T \|u(r)\|_{2+\alpha}^2dr <\infty,$$
thanks to Lemma \ref{a2}. On the other hand,
$$\int_\eps^T \|B(u^\mu(r-\mu),u(r))\|_\alpha^2 dr \leq  \sup_{t\in[\eps,T]} \|u(t)\|_{1+\alpha}^2 \|u\|^2_{L_2(\eps,T,V^{2+\alpha})}<\infty,$$
just taking $s_3=-\alpha$, $s_1=1+\alpha$ and $s_2=1+\alpha$ in Lemma \ref{l2}. 
\end{proof}

\section{Discrete and Continuous Dynamical Flows}\label{s4}

As pointed out above in the Introduction, in this paper we are interested in investigating the longtime behavior of the delayed Navier-Stokes equations (\ref{D2}). As in the study of the existence and uniqueness of solutions for (\ref{D2}), the analysis of its longtime behavior is based on the study of its corresponding linearized system. Hence, we first consider the linearized system on the whole positive real line and introduce its associated discrete flow $U$. Then, we establish a crucial relationship between $U$ and and the continuous flow $S$ related to (\ref{D2}), see (\ref{rs1}) below.

We consider the solution of (\ref{eq1}) to any compact interval. We can rewrite (\ref{eq1}) as
\begin{equation}\label{eq1b1}
  \left\{
  \begin{aligned}
    du_1(t)&+(\nu Au_1(t)+B(\psi(t), u_1(t)))dt=fdt ,\quad&t \in [0,\mu],&\\
    u_1(0)&=u_{0}.
  \end{aligned}
  \right.
\end{equation}
and consider generalizations of the above problem given for $k=2, 3, \cdots$ by
 \begin{equation}\label{eq1b2}
  \left\{
  \begin{aligned}
    du_k(t)&+(\nu Au_k(t)+B(u_{k-1}(t), u_k(t)))dt=fdt ,\quad&t \in [0,\mu],&\\
    u_k(0)&=u_{k-1}(\mu).
  \end{aligned}
  \right.
\end{equation}
It turns out that we can construct a sequence $\{(u_k)\}_{k\in \NN}\subset L_2(0,\mu,V^{1+\alpha})$ such that, for any $k\in \NN$, $u_k(\mu)\in V^\alpha$. Concatenating the elements of this sequence we can define the function $u$ given by
\begin{align}
\label{sol}
u(t)=\left\{\begin{array}{lll}
           u_0&\mbox{ if}& t=0,\\
        u_1(t)&\mbox{ if}& t\in [0,\mu],\\
     u_2(t-\mu)&\mbox{ if}& t\in [\mu,2\mu],\\
     \vdots&&\vdots \\
      u_k(t-(k-1)\mu)&\mbox{ if}& t\in [(k-1)\mu,T],
    \end{array}
    \right.
\end{align}
assuming that $t\in [0,T]$, with $T\in  ((k-1)\mu, k\mu]$. Therefore, we have constructed $u$ to be the solution of the linearized Navier-Stokes equations (\ref{eq1b1})-(\ref{eq1b2}) for $t\geq 0$. Due to the above construction, Lemma \ref{a1}, Lemma \ref{a2} and Lemma \ref{a3} can be also established for the solution (\ref{sol}).\\


Furthermore, we consider the dynamical system defined by $u^\mu$ and the dynamical system defined by $u$, and analyze the relationship between them. To be more precise, from now on we consider the two Hilbert spaces 
$$\mathcal{X}^\mu_\alpha=L_2(0,\mu,V^{1+\alpha})\times V^\alpha, \quad \mathcal Y^\mu_\alpha=L_2(-\mu,0,V^{1+\alpha})\times V^\alpha.$$

If $(x_1,x_2) \in \mathcal{X}^\mu_\alpha$ and $(y_1,y_2) \in \mathcal{Y}^\mu_\alpha$, the symbol
$$(x_1,x_2) \cong (y_1,y_2)$$
means that 
 \begin{equation*}
  \begin{aligned}
 x_1(\cdot)&=y_1(\cdot-\mu) \;{\text on \;} [0,\mu], \\
x_2&=y_2. 
  \end{aligned}
\end{equation*}
Therefore, $(\psi,u_0) \cong (\phi,u_0)$ means that $\psi(\cdot):=\phi(\cdot-\mu)$.\\

Notice that if $(\psi,u_0) \cong (\phi,u_0)$ then $u_1^\mu(t)=u_1(t),$ $t\in [0,\mu]$, which also implies
$$(u_1^\mu)_\mu(t-\mu)=u_1^\mu(t)=u_1(t),\; t\in [0,\mu],$$
that is,
$$(\psi,u_0) \cong (\phi,u_0) \; \Rightarrow (u_1,u_1(\mu)) \cong ((u_1^\mu)_\mu,u_1^\mu(\mu)).$$
By induction, for any $k\in \NN$ we obtain
\begin{equation}\label{con}
(\psi,u_0) \cong (\phi,u_0) \; \Rightarrow (u_k,u_k(\mu)) \cong ((u_k^\mu)_\mu,u_k^\mu(\mu)).
\end{equation}

Define for $n\in \NN$ and $t\in[0,T]$ the mappings $U(n,\cdot):\mathcal{X}^\mu_\alpha \to \mathcal{X}^\mu_\alpha$ and $S(t,\cdot):\mathcal Y^\mu_\alpha \rightarrow \mathcal Y^\mu_\alpha$, given respectively by
\begin{equation*}
U(n,(\psi,u_0))=(u|_{((n-1)\mu, n\mu]},u(n \mu)), \quad S(t,(\phi,u_0))=((u^\mu)_t, u^\mu(t)),
\end{equation*}
where $u$ is defined by \eqref{sol}  for $(\psi,u_0)\in  \mathcal X^\mu_\alpha$, and $u_t^\mu$ is the segment function defined by $(u^\mu)_t(s)=u^\mu(t+s)$, $s\in (-\mu,0)$, where $u^\mu$ is the weak solution to \eqref{D2} corresponding to the initial data $(\phi,u_0)\in \mathcal Y^\mu_\alpha$. Note that $U$ can be defined simply considering compositions of $U(1,\cdot)$ with itself. In fact, we can consider the one-step function $U(1, (\psi,u_0))=(u_1,u_1(\mu))$ and compose it with itself $n$ times, getting
\begin{align*}
\begin{split}
U(n, (\psi,u_0))=U(1,\cdot)\circ U(1,\cdot) \circ \cdots \circ U(1,(\psi,u_0)).
\end{split}
\end{align*}
It was proven, see  \cite{BGSch17}, that the discrete dynamical system $U(n,\cdot)$ is a continuous mapping on $\mathcal{X}^\mu_\alpha$ while $S(t,\cdot)$ is continuous on $\mathcal Y^\mu_\alpha$.\\

Now we can rewrite (\ref{con}) to establish the relationship between the discrete and continuous dynamical systems $U$ and $S$.  If $(\psi,u_0) \cong (\phi,u_0)$,
\begin{align*}
U(n,(\psi,u_0))=(u_n,u_n(\mu))\cong  ((u_{n}^\mu)_\mu, u^\mu_n(\mu))=((u^\mu)_{n\mu}, u^\mu(n\mu))= S(n\mu,(\phi,u_0)),
\end{align*}
or, in other words,
\begin{align}\label{rs1}
(\psi,u_0) \cong (\phi,u_0) \; \Rightarrow   U(n,(\psi,u_0)) \cong S(n\mu,(\phi,u_0)).
\end{align}

\begin{remark}\label{rn}
Our aim in the next section is to study the longtime behavior for (\ref{D2}). In particular, we will prove that there is an invariant ball $B_{\yY_\alpha^\mu}$ for the continuous semigroup $S(t)$, for $t\geq 0$. Notice that, given $t\geq 0$ there exists $n^\ast \in \NN$ such that $t\in [n^\ast \mu, (n^\ast+1)\mu]$, hence defining $\tau=t-n^\ast \mu \in [0,\mu]$, by the semigroup property
\begin{equation}\label{key1}
S(t,(\phi,u_0))=S(t-n^\ast \mu, S(n^\ast \mu,(\phi,u_0)))=S(\tau,((u_{n^\ast}^\mu)_\mu, u^\mu_{n^\ast}(\mu))).
\end{equation}
Therefore, combining (\ref{rs1}) and (\ref{key1}), it is clear that it is enough to restrict the investigation of the invariant ball for the continuous dynamical system $S$ to the interval $[0,\mu]$. More precisely, the invariance of the ball $B_{\yY_\alpha^\mu}$ will follow in two steps: first, we will find an invariant ball $B_{\xX_\alpha^\mu}$ for the discrete dynamical flow $U$, and then we study the invariance of a ball for $S$ on $[0,\mu]$.

\end{remark}

\section{Longtime behavior for the linearized equation}

As we have said in the Introduction, the main goal of this paper is to investigate the existence of an attractor for the delayed Navier-Stokes equations (\ref{D2}). To do that, we are going to use the relationship (\ref{rs1}). To be more precise, in a first step we consider the discrete dynamical system $U$ and look for the existence of an discrete attractor associated to $U$. The existence of this discrete attractor rests upon the invariance of a ball $B\in \mathcal X^\mu_\alpha$ for $U$ (see Lemma \ref{ib} below) and suitable compact embeddings of some spaces (see Lemma \ref{t1b}).\\

To simplify the presentation, we identify $U(\psi, u_0)$ with $U(1,(\psi, u_0))$. Also, since we believe that confusion is not possible, we drop the subindex and represents the solution by $u$ instead of by $u_1$.

\begin{lemma}\label{ib}
Consider $(\psi,u_0)\in \mathcal X^\mu_\alpha$ and $f\in V^{\alpha-1}$ and assume that the viscosity is large enough. Then for $U(\psi,u_0)=(u_1,u_1(\mu))$ defined in Section \ref{s4} we have
$$U (B_{\mathcal X_\alpha^\mu} (R;\rho))\subset B_{\mathcal X_\alpha^\mu} (R;\rho),$$
where $B_{\mathcal X_\alpha^\mu} (R;\rho):=B_{L_2(0,\mu,V^{1+\alpha})}(0,R)\times B_{V^\alpha}(0,\rho)$, with $R$ and $\rho$ defined by (\ref{cond})-(\ref{rho}) below.
\end{lemma}

\begin{proof} By assumption the viscosity is large enough, hence we can find $R>0$ such that 
\begin{align}\label{cond}
-\frac{\nu \lambda \mu}{2}+ \frac{c^2 R^2}{\nu} <-\ln (2), 
\end{align}
and 
\begin{align}\label{cond2}
\frac{8}{\nu^2 \lambda} \|f\|^2_{\alpha-1} e^{\frac{\nu \lambda \mu}{2}}\left(   \frac{2}{\nu}+  \frac{2c^2R^2}{\nu^2}  (  e^{\frac{c^2 R^2}{\nu}}+ \frac{1}{2})+ \frac{\lambda \mu e^{-\frac{\nu \lambda \mu}{2}} }{2}  \right)\leq \frac{R^2}{2},
\end{align}

where $c$ is the positive constant determined in Lemma \ref{l2} and $\lambda$ denotes the first eigenvalue of $A$. Let us define $\rho>0$ given by 
\begin{align}\label{rho}
\rho^2= \frac{8}{\nu^2 \lambda} \|f\|^2_{\alpha-1} e^{\frac{\nu \lambda \mu}{2}}.
\end{align}

We denote by $pr_i(\cdot)$, $i=1,2$, the projection into the corresponding component.

We start proving that 
\begin{align}\label{p1}
pr_2 U (B_{\mathcal X_\alpha^\mu} (R;\rho))\subset B_{V^\alpha}(0,\rho),
\end{align}
for which we need to prove that if $\|\psi\|_{L_2(0,\mu,V^{1+\alpha})}^2\leq R^2$ and $\|u_0\|_\alpha^2\leq \rho^2$, then $\|u(\mu)\|_\alpha^2\leq \rho^2$.
For $t\in [0,\mu]$, considering the scalar product with $A^{\alpha} u(t)$ in $V^0$,  
 we have
\begin{align*}
  \frac{d}{dt}\|u(t)\|_{\alpha}^2+2\nu\|u(t)\|_{1+\alpha}^2&\le 2c\|\psi(t)\|_{1+\alpha }\|u(t)\|_{\alpha }\|u(t)\|_{1+\alpha} +\frac{2}{\nu}\|f\|_{\alpha-1}^2+\frac{\nu}{2}\|u(t)\|_{1+\alpha}^2\\
  &\le \frac{c^2}{\nu}\|\psi(t)\|_{1+\alpha }^2\|u(t)\|_{\alpha }^2+\nu\|u(t)\|_{1+\alpha}^2+ \frac{2}{\nu}\|f\|_{\alpha-1 }^2
  +\frac{\nu}{2}\|u(t)\|_{1+\alpha}^2,
\end{align*}
where we have applied Lemma \ref{l2} taking $s_3=-\alpha$, $s_1=1+\alpha$ and $s_2=\alpha$. Hence, applying Gronwall's lemma, 
\begin{align}\label{f}
\|u(t)\|_\alpha^2\leq \|u_0\|_\alpha^2 e^{-\frac{\nu \lambda t}{2}+ \frac{c^2}{\nu}\int_0^t\|\psi(r)\|_{1+\alpha}^2dr}+\frac{2}{\nu }\|f\|_{\alpha-1}^2 \int_0^t e^{-\frac{\nu \lambda (t-s)}{2}+ \frac{c^2}{\nu}\int_s^t\|\psi(r)\|_{1+\alpha}^2dr} ds
\end{align}
and in particular we obtain
\begin{align*}
\|u(\mu)\|_\alpha^2\leq \|u_0\|_\alpha^2 e^{-\frac{\nu \lambda \mu}{2}+ \frac{c^2 R^2}{\nu}}+\frac{4}{\nu^2 \lambda }\|f\|_{\alpha-1}^2 e^{\frac{c^2 R^2}{\nu}}(1-e^{-\frac{\nu \lambda \mu}{2}}).
\end{align*}
Since $-\frac{\nu \lambda \mu}{2}+ \frac{c^2 R^2}{\nu} <-\ln (2)<0$, we have that 
$e^{ \frac{c^2 R^2}{\nu} }\leq e^{\frac{\nu \lambda \mu}{2}}$.
Hence,
\begin{align*}
\|u(\mu)\|_\alpha^2\leq \frac{1}{2} \|u_0\|_\alpha^2 +\frac{4}{\nu^2 \lambda }\|f\|_{\alpha-1}^2 e^{\frac{\nu \lambda \mu}{2}},
\end{align*}

thus, if $\|u_0\|_\alpha^2\leq \rho^2$ with $\rho$ defined by (\ref{rho}), we get that $\|u(\mu)\|_\alpha^2 \leq \rho^2$, thus (\ref{p1}) is proved.

Let us prove now that 
\begin{align}\label{p2}
pr_1 U (B_{\mathcal X_\alpha^\mu} (R;\rho))\subset B_{L_2(0,\mu,V^{1+\alpha})}(0,R).
\end{align}
Notice that from the previous estimates we also obtain
\begin{align*}
\frac{\nu}{2} \int_0^\mu \|u(r)\|_{1+\alpha}^2 dr&
\le \|u_0\|_\alpha^2+ \frac{c^2}{\nu} \int_0^\mu \|\psi(r)\|_{1+\alpha }^2\|u(r)\|_{\alpha }^2dr+ \frac{2\mu}{\nu}\|f\|_{\alpha -1}^2,
\end{align*}
and from (\ref{f}) we also know that
\begin{align*}
\sup_{t\in [0,\mu]}\|u(t)\|_\alpha^2 &\leq \|u_0\|_\alpha^2  e^{\frac{c^2 R^2}{\nu}}  \sup_{t\in [0,\mu]} e^{-\frac{\nu \lambda t}{2}}+\frac{2}{\nu }\|f\|_{\alpha-1}^2   e^{\frac{c^2 R^2}{\nu}} \sup_{t\in [0,\mu]} \int_0^t e^{-\frac{\nu \lambda (t-s)}{2}} ds\\
&\leq \|u_0\|_\alpha^2  e^{\frac{c^2 R^2}{\nu}} +\frac{4}{\nu^2 \lambda }\|f\|_{\alpha-1}^2   e^{\frac{c^2 R^2}{\nu}}    \sup_{t\in [0,\mu]} (1-e^{-\frac{\nu \lambda t}{2}})\\
&\leq \|u_0\|_\alpha^2  e^{\frac{c^2 R^2}{\nu}} +\frac{4}{\nu^2 \lambda }\|f\|_{\alpha-1}^2   e^{\frac{c^2 R^2}{\nu}} (1-e^{-\frac{\nu \lambda \mu}{2}})\\
&\leq \rho^2 (  e^{\frac{c^2 R^2}{\nu}}+ \frac{1}{2}).
\end{align*}

Therefore, by (\ref{cond2}),
\begin{align}\label{ae}
\begin{split}
\int_0^\mu \|u(r)\|_{1+\alpha}^2 dr&
\le \frac{2}{\nu} \rho^2+ \frac{2c^2R^2}{\nu^2}  \rho^2 (  e^{\frac{c^2 R^2}{\nu}}+ \frac{1}{2}) + \frac{4\mu}{\nu^2}\|f\|_{\alpha-1 }^2\\
&\leq \frac{2}{\nu} \rho^2+ \frac{2c^2R^2}{\nu^2} \rho^2 (  e^{\frac{c^2 R^2}{\nu}}+ \frac{1}{2})+ \frac{\lambda \mu e^{-\frac{\nu \lambda \mu}{2}} }{2}  \rho^2 \\
&\leq  \rho^2\left(   \frac{2}{\nu}+  \frac{2c^2R^2}{\nu^2}  (  e^{\frac{c^2 R^2}{\nu}}+ \frac{1}{2})+ \frac{\lambda \mu e^{-\frac{\nu \lambda \mu}{2}} }{2}  \right).
\end{split}
\end{align}
We conclude that choosing $\nu$ big enough, we arrive at 
\begin{align*}
\|u\|_{L_2(0,\mu,V^{1+\alpha})}^2 dr \leq \frac{R^2}{2}\leq R^2,
\end{align*}
hence (\ref{p2}) is proved and the proof is finished.
\end{proof}

\begin{remark} \label{lemmar} 
 In the expression (\ref{ae}), we take the left hand side to be smaller than $\frac{R^2}{2}$ but not directly smaller than $R^2$. The reason is that this choice will help us later to show the invariance of a ball for the continuous dynamical system $S$, see Section \ref{s5}. Anyway, $B_{\mathcal X_\alpha^\mu} (R;\rho)$ is invariant for $U$ because starting in $(\psi,u_0) \in B_{\mathcal X_\alpha^\mu} (R;\rho)$ we know that 
$$U (\psi,u_0) \in B_{\mathcal X_\alpha^\mu} \bigg(\frac{R}{\sqrt{2}};\rho\bigg) \subset B_{\mathcal X_\alpha^\mu} (R;\rho),$$
where $B_{\mathcal X_\alpha^\mu} (\frac{R}{\sqrt{2}};\rho):=B_{L_2(0,\mu,V^{1+\alpha})}\bigg(0,\frac{R}{\sqrt{2}}\bigg)\times B_{V^\alpha}(0,\rho)$.\\
\end{remark}

Now we want to establish the existence of a unique discrete attractor associated to $U$. To do that, we are using the following lemma, whose proof can be found in Vishik and Fursikov \cite{FurVis} Chapter IV Theorem 4.1. 

\begin{lemma}\label{t1b}
The space $L_2(s,t,V^{2+\alpha})\cap C^\beta([s,t],V^{\alpha})$ is compactly embedded into $L_2(s,t,V^{\alpha})\cap C([s,t],V^{\alpha})$.\\
\end{lemma}

As a consequence of the previous results, we can establish one of the main theorems of this article:
\begin{theorem}\label{t2}
Assume that the viscosity is large enough. Then the discrete dynamical system U associated to the linearized 3D Navier-Stokes equations possesses a local attractor $\aA$.


\end{theorem}
\begin{proof}
According to Lemma \ref{ib} and Remark \ref{lemmar}, since the viscosity is large enough we can find $R$ such that (\ref{cond}) and (\ref{cond2}) hold true, which imply that $U(n,B_{\mathcal X_\alpha^\mu} (R;\rho))\subset B_{\mathcal X_\alpha^\mu} (R;\rho)$, that is, $B_{\mathcal X_\alpha^\mu} (R;\rho)$ is a forward invariant ball, with $\rho$ given by (\ref{rho}).

On the other hand, due to the extra regularity of the solution given by Lemma \ref{a1}, Lemma \ref{a2} and Lemma \ref{a3} (for the solution of the linearized equation), in virtue of Lemma \ref{t1b}, we know that $U(n,B_{\mathcal X_\alpha^\mu} (R;\rho))$ is relatively compact for $n\geq 2$. 
Now defining 
$$\mathcal K:=\overline{U(2, B_{\mathcal X_\alpha^\mu} (R;\rho))}^{\mathcal X^\mu_\alpha}$$ we have that $\mathcal K$ is a forward invariant compact set, hence $U$ possesses a unique local attractor $\aA$ (for a comprehensive presentation of the concept of attractors we refer to the monographs by Babin and Vishik \cite{BaVis92}, Hale
\cite{Hale} or Temam \cite{Tem97}).

\end{proof}

\begin{remark}
Observe that $B_{\mathcal X_\alpha^\mu} (R;\rho)$ is not an absorbing ball but an invariant ball. Hence, we know that $B_{\mathcal X_\alpha^\mu} (R;\rho)$ absorbs elements of itself, but not of any bounded set in $\mathcal X_\alpha^\mu$. Hence, the localness is related to the fact that the initial condition $(\psi,u_0)$ must belong to the ball $B_{\mathcal X_\alpha^\mu} (R;\rho)$.

\end{remark}

\subsection{Single point local attractor}

Now we are interested in finding sufficient conditions that ensure that the local attractor $\mathcal A$ associated to the discrete dynamical system $U$ is a single point.
\begin{theorem}\label{sp}
Assume that the viscosity is large enough. Then the local attractor $\mathcal A$ of Theorem \ref{t2}
consists of a single point.
\end{theorem}
\begin{proof}
Assume that $u_1$ and $u_2$ are two weak solutions to (\ref{eq1}) with initial conditions given, respectively, by $(\psi_1, u_{0,1})$, $(\psi_2, u_{0,2})\in B_{\mathcal X_\alpha^\mu} (R;\rho)$, the forward invariant ball of Lemma \ref{ib}. Then, the difference $u_1-u_2$ verifies

\begin{align*}
\frac{d}{dt}(u_1(t)-u_2(t))+\nu A(u_1(t)-u_2(t))&=B(\psi_1(t),u_1(t))-B(\psi_2(t),u_2(t))\\
&=B(\psi_1(t)-\psi_2(t),u_1(t))+B(\psi_2(t),u_1(t)-u_2(t)),
\end{align*}

then, multiplying by $A^{\alpha}(u_1(t)-u_2(t))$ in $V^0$ we have

\begin{align*}
 \frac{d}{dt}\|u_1(t)-u_2(t)\|_{\alpha}^2+2\nu\|u_1(t)-u_2(t)\|_{1+\alpha}^2&\le 2c\|\psi_1(t)-\psi_2(t)\|_{1+\alpha }\|u_1(t)\|_{\alpha }\|u_1(t)-u_2(t)\|_{1+\alpha} \\
  &+2c\|\psi_2(t)\|_{1+\alpha }\|u_1(t)-u_2(t)\|_{\alpha }\|u_1(t)-u_2(t)\|_{1+\alpha}\\
  &\le \frac{2c^2}{\nu}\|\psi_1(t)-\psi_2(t)\|_{1+\alpha }^2\|u_1(t)\|_{\alpha }^2+\frac{\nu}{2}\|u_1(t)-u_2(t)\|_{1+\alpha}^2\\&
  + \frac{2c^2}{\nu}\|\psi_2(t)\|_{1+\alpha }^2\|u_1(t)-u_2(t)\|_{\alpha }^2+\frac{\nu}{2}\|u_1(t)-u_2(t)\|_{1+\alpha}^2,
 \end{align*}
(considering in Lemma \ref{l2} $s_1=1+\alpha$, $s_2=\alpha-1=-s_3$ ), which gives 
\begin{align*}
 \frac{d}{dt}\|u_1(t)-u_2(t)\|_{\alpha}^2+\nu\|u_1(t)-u_2(t)\|_{1+\alpha}^2  &\le \frac{2c^2}{\nu}\|\psi_1(t)-\psi_2(t)\|_{1+\alpha }^2\|u_1(t)\|_{\alpha }^2+ \frac{2c^2}{\nu}\|\psi_2(t)\|_{1+\alpha }^2\|u_1(t)-u_2(t)\|_{\alpha }^2
 \end{align*}

and since $\psi_2 \in B_{L_2(0,\mu,V^{1+\alpha})}(0,R)$, applying Gronwall's lemma,
\begin{align*}
\|u_1(t)-u_2(t)\|_{\alpha}^2&\le e^{-\lambda \nu \mu+\frac{2c^2 R^2}{\nu}} \|u_{0,1}-u_{0,2}\|_\alpha^2+\frac{2c^2}{\nu}\int_0^\mu e^{-\lambda \nu (\mu-s)+\frac{2c^2 R^2}{\nu}}   \|\psi_1(s)-\psi_2(s)\|_{1+\alpha }^2\|u_1(s)\|_{\alpha }^2 ds,
 \end{align*}
and because we know that $U(1,B) \subset B$, where $B=B_{\mathcal X_\alpha^\mu} (R;\rho)$, then $$\sup_{s\in [0,\mu]}\|u_1(s)\|_{\alpha }^2 \leq \rho^2$$
and this implies
\begin{align*}
\|u_1(t)-u_2(t)\|_{\alpha}^2&\le e^{-\lambda \nu \mu+\frac{2c^2 R^2}{\nu}} \|u_{0,1}-u_{0,2}\|_\alpha^2+\frac{2c^2 \rho^2}{\lambda \nu^2} (1-e^{-\lambda \nu \mu}) e^{\frac{2c^2 R^2}{\nu}}  \|\psi_1-\psi_2\|_{L_2(0,\mu,V^{1+\alpha})}^2.
 \end{align*}
 
 Furthermore, 
 \begin{align*}
 \nu\|u_1(t)-u_2(t)\|_{1+\alpha}^2&\le \frac{2c^2}{\nu}\|\psi_1(t)-\psi_2(t)\|_{1+\alpha }^2\|u_1(t)\|_{\alpha }^2
+ \frac{2c^2}{\nu} \|\psi_2(t)\|_{1+\alpha }^2\|u_1(t)-u_2(t)\|_{\alpha }^2,
 \end{align*}
 and thus
 \begin{align*}
\nu \int_0^\mu \|u_1(r)-u_2(r)\|_{1+\alpha}^2dr&\le  \|u_{0,1}-u_{0,2}\|_\alpha^2+ \frac{2c^2}{\nu}\int_0^\mu \|\psi_1(r)-\psi_2(r)\|_{1+\alpha }^2\|u_1(r)\|_{\alpha }^2dr\\
&+\frac{2c^2}{\nu} \int_0^\mu  \|\psi_2(r)\|_{1+\alpha }^2\|u_1(r)-u_2(r)\|_{\alpha }^2dr\\
& \le  \|u_{0,1}-u_{0,2}\|_\alpha^2+ \frac{2c^2\rho^2}{\nu} \|\psi_1-\psi_2\|_{L_2(0,\mu,V^{1+\alpha})}^2\\
&+\frac{2c^2}{\nu}\sup_{s\in [0,\mu]} \|u_1(s)-u_2(s)\|_{\alpha }^2  \int_0^\mu \|\psi_2(r)\|_{1+\alpha }^2dr.
\end{align*}
Now, if we divide by $\nu$ we obtain that
 \begin{align*}
\int_0^\mu \|u_1(r)-u_2(r)\|_{1+\alpha}^2dr & \le  \frac{1}{\nu} \|u_{0,1}-u_{0,2}\|_\alpha^2+ \frac{2c^2\rho^2}{\nu^2} \|\psi_1-\psi_2\|_{L_2(0,\mu,V^{1+\alpha})}^2\\
&+\frac{2c^2R^2}{\nu^2}\sup_{s\in [0,\mu]} \|u_1(s)-u_2(s)\|_{\alpha }^2.
\end{align*}

Hence, adding the two estimates yields
\begin{align*}
\|U(1, (\psi_1,u_{0,1}))-U(1, (\psi_2,u_{0,2}))\|_{\mathcal X^\mu_\alpha}^2&\le \big( \bigg(\frac{2c^2R^2}{\nu^2}+1\bigg) \e^{-\lambda \nu \mu+\frac{2c^2 R^2}{\nu^2}} +\frac{1}{\nu}\big) \|u_{0,1}-u_{0,2}\|_\alpha^2\\
&+  \big( \frac{1}{\lambda} \bigg(\frac{2c^2R^2}{\nu^2}+1\bigg)(1-e^{-\lambda \nu \mu}) e^{\frac{2c^2 R^2}{\nu}}+1\big) \frac{2c^2 \rho^2 }{\nu^2} 
\|\psi_1-\psi_2\|_{L_2(0,\mu,V^{1+\alpha})}^2.
\end{align*}

Now, we choose $\nu$ large enough such that 
\begin{align*}
&\bigg(\frac{2c^2R^2}{\nu^2}+1\bigg)\e^{-\lambda \nu \mu+\frac{2c^2 R^2}{\nu^2}} +\frac{1}{\nu}<\frac{1}{2}\\
& \frac{1}{\lambda} \bigg(\frac{2c^2R^2}{\nu^2}+1\bigg) (1-e^{-\lambda \nu \mu}) e^{\frac{2c^2 R^2}{\nu}}+1\big) \frac{2c^2 \rho^2 }{\nu^2}  <\frac{1}{2}.
\end{align*}
Altogether, the above inequalities in turn imply
\begin{align*}
\|U(1, (\psi_1,u_{0,1}))-U(1, (\psi_2,u_{0,2}))\|_{\mathcal X^\mu_\alpha}^2&<\frac{1}{2}(\|u_{0,1}-u_{0,2}\|_\alpha^2+  \|\psi_1-\psi_2\|_{L_2(0,\mu,V^{1+\alpha})}^2)
 \end{align*}
 provided that $\nu$ is sufficiently large. By repeating the same arguments, for any $n\in \NN$ we obtain
\begin{align*}
\|U(n, (\psi_1,u_{0,1}))-U(n, (\psi_2,u_{0,2}))\|_{\mathcal X^\mu_\alpha}^2&<\frac{1}{2^n}(\|u_{0,1}-u_{0,2}\|_\alpha^2+  \|\psi_1-\psi_2\|_{L_2(0,\mu,V^{1+\alpha})}^2).
 \end{align*} 
 In other words, due to the invariance property of the attractor $\aA$ we have obtained
\begin{align*}
\sup_{y_1,y_2\in \aA}\|y_1-y_2\|_{\mathcal X^\mu_\alpha}^2&<\frac{1}{2^n}\sup_{x_1,x_2\in \aA}\|x_1-x_2\|_{\mathcal X^\mu_\alpha}^2, \end{align*}
and, as the right--hand side tends to zero, this implies that $\aA$ is a single point attractor for $U$.

\end{proof}

\section{Longtime behavior for the delayed Navier-Stokes equations}\label{s5}

Now we are in position to establish our main result: the existence and uniqueness of a local attractor for the continuous dynamical system $S$. As mentioned before, the results will be based on the relationship (\ref{rs1}) and the fact that $U$ has a unique local attractor $\aA$. 

Let us recall that
$$B_{\mathcal X_\alpha^\mu} (R;\rho):=B_{L_2(0,\mu,V^{1+\alpha})}(0,R)\times B_{V^\alpha}(0,\rho),$$
$$B_{\mathcal X_\alpha^\mu} \bigg(\frac{R}{\sqrt{2}};\rho \bigg ):=B_{L_2(0,\mu,V^{1+\alpha})}\bigg(0,\frac{R}{\sqrt{2}}\bigg)\times B_{V^\alpha}(0,\rho)$$
with $R$ and $\rho$ defined by (\ref{cond})-(\ref{rho}). Define now
\begin{align}
B_{\mathcal Y_\alpha^\mu} (R;\rho)&:=B_{L_2(-\mu,0,V^{1+\alpha})}(0,R)\times B_{V^\alpha}(0,\rho),\label{b1}\\
B_{\mathcal Y_\alpha^\mu} \bigg(\frac{R}{\sqrt{2}};\rho \bigg )&:=B_{L_2(-\mu,0,V^{1+\alpha})}\bigg(0,\frac{R}{\sqrt{2}}\bigg)\times B_{V^\alpha}(0,\rho).\label{b2}
\end{align}

We would like to show that $B_{\mathcal Y_\alpha^\mu} (R;\rho)$ is a forward invariant ball for $S$. First of all, we prove the following result for the time interval $[0,\mu]$.

\begin{lemma}\label{ll}
Assume the conditions (\ref{cond}) and (\ref{cond2}) and consider $(\phi,u_0)\in B_{\mathcal Y_\alpha^\mu} \bigg(\frac{R}{\sqrt{2}};\rho \bigg )$. Then for $t\in [0,\mu]$ yields
$$S(t,(\phi,u_0))=((u_{1}^\mu)_t,u_1^\mu(t))\in B_{\mathcal Y_\alpha^\mu} (R;\rho).$$ 
\end{lemma}

\begin{proof}

We only sketch the proof since it is very similar to the one of Lemma \ref{ib}.  Indeed, following the same steps than the ones of Lemma \ref{ib} we can obtain that  $\|u_1^\mu\|^2_{L_2(0,\mu,V^{1+\alpha})} \leq \frac{R^2}{2}$. Hence if $(\phi,u_0)\in B_{\mathcal Y_\alpha^\mu} \bigg(\frac{R}{\sqrt{2}};\rho \bigg )$ we obtain that 
$$\|(u_{1}^\mu)_t\|^2_{L_2(-\mu,0,V^{1+\alpha})} \leq \|\phi\|^2_{L_2(-\mu,0,V^{1+\alpha})}+\|u_1^\mu\|^2_{L_2(0,\mu,V^{1+\alpha})}\leq \frac{R^2}{2}+\frac{R^2}{2}=R^2.$$

For the second component, we should prove that 
$$\|u_1^\mu(t)\|_{\alpha}^2\leq \rho^2.$$

We know that

\begin{align*}
  \frac{d}{dt}\|u_1^\mu(t)\|_{\alpha}^2+2\nu\|u_1^\mu(t)\|_{1+\alpha}^2&\le 2c\|\phi(t-\mu)\|_{1+\alpha }\|u_1^\mu(t)\|_{\alpha }\|u_1^\mu(t)\|_{1+\alpha} +\frac{2}{\nu}\|f\|_{\alpha-1}^2+\frac{\nu}{2}\|u_1^\mu(t)\|_{1+\alpha}^2\\
  &\le \frac{c^2}{\nu}\|\phi(t-\mu)\|_{1+\alpha }^2\|u_1^\mu(t)\|_{\alpha }^2+\nu\|u_1^\mu(t)\|_{1+\alpha}^2+ \frac{2}{\nu}\|f\|_{\alpha-1 }^2
  +\frac{\nu}{2}\|u_1^\mu(t)\|_{1+\alpha}^2,
\end{align*}
hence, applying Gronwall's lemma, 
\begin{align*}\label{f}
\|u_1^\mu(t)\|_\alpha^2&\leq \|u_0\|_\alpha^2 e^{-\frac{\nu \lambda t}{2}+ \frac{c^2}{\nu}\int_0^t\|\phi(r-\mu)\|_{1+\alpha}^2dr}+\frac{2}{\nu }\|f\|_{\alpha-1}^2 \int_0^t e^{-\frac{\nu \lambda (t-s)}{2}+ \frac{c^2}{\nu}\int_s^t\|\phi(r-\mu)\|_{1+\alpha}^2dr} ds\\
&\leq \rho^2 e^{-\frac{\nu \lambda \mu}{2}+ \frac{c^2R^2}{2\nu}}+\frac{2}{\nu}\|f\|_{\alpha-1}^2e^{ \frac{c^2R^2}{2\nu}} 
\int_0^t e^{-\frac{\nu \lambda (t-s)}{2}}ds\\
&\leq \rho^2 e^{-\frac{\nu \lambda \mu}{2}+ \frac{c^2R^2}{2\nu}}+\frac{4}{\nu^2\lambda}\|f\|_{\alpha-1}^2e^{ \frac{c^2R^2}{2\nu}} \\
&\leq \rho^2 e^{-\frac{\nu \lambda \mu}{2}+ \frac{c^2R^2}{\nu}}e^{-\frac{c^2R^2}{2\nu}}+\frac{4}{\nu^2\lambda}\|f\|_{\alpha-1}^2e^{ \frac{\nu\lambda\mu}{2}} e^{ \frac{c^2R^2}{2\nu}-\frac{\nu \lambda \mu}{2}}.
\end{align*}

Now, using the definition of $\rho$ given by (\ref{rho}) and assumption \eqref{cond}, we get that
\begin{equation}
\|u_1^\mu(t)\|_\alpha^2 \leq \frac{1}{2} \rho^2+\frac{1}{4} \rho^2 \leq \rho^2. \end{equation}

\end{proof}

In the next result we prove that $B_{\mathcal Y_\alpha^\mu} (R;\rho)$ is an invariant ball for $S$.

\begin{lemma}\label{ib2}
Assume the viscosity is sufficiently large. Then $B_{\mathcal Y_\alpha^\mu} (R;\rho)$ defined by (\ref{b1}) with $R$ and $\rho$ defined in Lemma \ref{ib}, is forward invariant for $S$, that is, for every $t\geq 0$ and $(\phi, u_0)\in B_{\mathcal Y_\alpha^\mu} (R;\rho)$, we have 
$$S (t, (\phi, u_0))\in B_{\mathcal Y_\alpha^\mu} (R;\rho).$$
\end{lemma}

\begin{proof}
First of all, notice that if $(\psi,u_0) \cong (\phi,u_0)$ and $(\phi,u_0)\in B_{\mathcal Y_\alpha^\mu} (R;\rho)$ then trivially $(\psi,u_0)\in B_{\mathcal X_\alpha^\mu} (R;\rho)$,  the invariant ball obtained in Lemma \ref{ib}.

Given $t\geq 0$ there exists $n^\ast \in \NN$ such that $t\in [n^\ast \mu, (n^\ast+1)\mu]$ and (\ref{key1}) holds true, that is 
\begin{equation*}
S(t,(\phi,u_0))=S(t-n^\ast \mu, S(n^\ast \mu,(\phi,u_0)))=S(\tau,((u_{n^\ast}^\mu)_\mu, u^\mu_{n^\ast}(\mu))),
\end{equation*}
with $\tau=t-n^\ast \mu \in [0,\mu]$. 

On the other hand, (\ref{rs1}) reads as
\begin{align*}
U(n^\ast,(\psi,u_0))=(u_{n^\ast},u_{n^\ast}(\mu)) \cong ((u_{n^\ast}^\mu)_\mu, u^\mu_{n^\ast}(\mu))=S(n^\ast \mu,(\phi,u_0)) .
\end{align*}
Thanks to the invariance of $B_{\mathcal X_\alpha^\mu} (R;\rho)$ under $U$,  $(\psi,u_0) \in B_{\mathcal X_\alpha^\mu} (R;\rho)$ implies $(u_{n^\ast},u_{n^\ast}(\mu))\in B_{\mathcal X_\alpha^\mu} (\frac{R}{\sqrt{2}};\rho)$, see Remark \ref{lemmar}. This statement together with (\ref{key1}) imply that, in order to see the invariance of $B_{\mathcal Y_\alpha^\mu} (R;\rho)$ under $S$ it is enough to prove that given $(\phi,u_0)\in  B_{\mathcal Y_\alpha^\mu} (\frac{R}{\sqrt{2}};\rho)$ then $S(t,(\phi,u_0))=((u_{1}^\mu)_t,u_1^\mu(t))\in B_{\mathcal Y_\alpha^\mu} (R;\rho)$, when $t\in [0,\mu]$. Now, in virtue of  Lemma \ref{ll}, the result is proven.

\end{proof}

As an immediate consequence, we obtain the existence of the attractor $\aA^\mu$ for $S$:
\begin{theorem}\label{t2b}
Assume the viscosity is sufficiently large. Then the continuous dynamical system $S$ possesses a unique local attractor $\mathcal A^\mu$. Furthermore, under the same choice of the viscosity as in Theorem \ref{sp}, the attractor $\mathcal A^\mu$ of $S$ consists of a single point.
\end{theorem}

\begin{proof}
The first part of the statement follows as a consequence of the previous results, hence we only sketch the proof of the single point local attractor. Thanks to Theorem \ref{sp} we know that $\mathcal A=\{(\hat \psi, \hat u_0)\}.$ We want to prove that $\mathcal A^\mu=\{(\hat \phi, \hat u_0)\},$ that is, $\mathcal A^\mu$ is a single point attractor that in addition is linked to the single point  attractor $\mathcal A$ by the relation $ (\hat \psi, \hat u_0)\cong (\hat \phi, \hat u_0).$

Consider precisely $(\hat \psi, \hat u_0)\in \xX^\mu_\alpha$ and its corresponding pair $(\hat \phi, \hat u_0)\in \yY^\mu_\alpha$. Then
$$S(n\mu,(\hat \phi, \hat u_0))\cong U(n,(\hat \psi, \hat u_0))=(\hat \psi, \hat u_0) \cong (\hat \phi, \hat u_0),$$
therefore $(\hat \phi, \hat u_0)\in \aA^\mu$.  Conversely, take any pair $(\tilde \phi, \tilde u_0)\in \aA^\mu$ and the corresponding $(\tilde \psi, \tilde u_0)\in \xX^\mu_\alpha$ such that $(\tilde \psi, \tilde u_0)\cong (\tilde \phi, \tilde u_0)$. Then 
$$U(n,(\tilde \psi, \tilde u_0)) \cong S(n\mu ,(\tilde \phi, \tilde u_0))=(\tilde \phi, \tilde u_0) \cong (\tilde \psi, \tilde u_0)$$
thus $(\tilde \psi, \tilde u_0)\in \aA$, that is to say, $(\tilde \psi, \tilde u_0)=(\hat \psi, \hat u_0),$ and then $\aA^\mu =\{(\hat \phi, \hat u_0)\}.$

\end{proof}

{\bf Aknowledgement:} This paper was partially finished while H. Bessaih was visiting the Faculty of Mathematics of the University of Sevilla and the department of Mathematics of Jena and she would like to thank them for a very warm hospitality.  H. Bessaih was partially supported by Simons Foundation grant: 582264, M.J. Garrido-Atienza was partially supported by grant PGC2018-096540-I00.
Finally, the authors would like to thank Prof. Bj\"orn Schmalfu\ss \, for very stimulating discussions on the topic of the paper


\begin{thebibliography}{99}

\bibitem{BaVis92}
A.~V. Babin and M.~I. Vishik, {\em Attractors of evolution equations},
Studies in Mathematics and its Applications, 25, North-Holland Publishing Co., Amsterdam, 1992.

\bibitem{BGSch17} H. Bessaih, M. J. Garrido-Atienza, B. Schmalfu\ss, {On 3D Navier-Stokes equations: regularization and uniqueness by delays}, {\em Physica D: Nonlinear Phenomena}, 376/377 (2018), 228--237.


\bibitem{Fursikov}
A. V. Fursikov, {\em Optimal Control of Distributed Systems. Theory and Applications}. Translations of Mathematical Monographs, 187. American Mathematical Society, Providence, RI, 2000.

\bibitem{Hale}
J.~K. Hale, {\em Asymptotic behavior of dissipative systems},
Mathematical Surveys and Monographs, 25. American Mathematical Society, Providence, RI, 1988.

  \bibitem {Temam}
R. Temam,
{\it  Navier-Stokes equations. Theory and Numerical Analysis},
North-Holland Publishing Co., Amsterdam-New York, 1977.

\bibitem{temamP}
R. Temam, {\it Navier-Stokes equations and nonlinear functional analysis}, Second edition.
CBMS-NSF Regional Conference Series in Applied Mathematics, 66.
Society for Industrial and Applied Mathematics (SIAM), Philadelphia, PA, 1995.

\bibitem{Tem97}
R. Temam, {\em Infinite-dimensional dynamical systems in mechanics and physics},
Second edition, Applied Mathematical Sciences, 68. Springer-Verlag, New York, 1997.

\bibitem{FurVis}
M.~I. Vishik and A.~V. Fursikov,
{\em Mathematical Problems of Statistical Hydromechanics},
Kluwer Academic Publishers, 1988.

\end{thebibliography}
\end{document}